\documentclass[11pt,twoside]{article}

    \usepackage[left=1in,top=1in,right=1in,bottom=1in,head=.1in]{geometry}
    \usepackage{graphicx}
    \usepackage{xy}    
    \usepackage{bm}
    \xyoption{all}
    \usepackage{url}
    \usepackage[linktocpage]{hyperref}
    \usepackage{styleset}
    \usepackage{mymacros}
    \renewcommand\thefootnote{\arabic{footnote}}
\newcommand{\vol}{\operatorname{vol}}
    \title  {Equivariant hyperbolization of \,$3$-manifolds \\ \qquad 
    via homology cobordisms}

    \author{Dave Auckly\footnote{Supported by an AIM SQuaRE 
    $^2$\,Partially supported by Simons Foundation grant 585139, and National Science Foundation Grant 1952755
    $^3$\,Supported by BK21 PLUS SNU Mathematical Sciences Division and NRF Grant 2015R1D1A1A01059318  
    $^4$\,Partially supported by National Science Foundation Grant 1506328}$^{,2}$, Hee Jung Kim$^{1,3}$, Paul Melvin$^1$ and Daniel Ruberman$^{1,4}$
    }

    \address {Kansas State University, Manhattan KS 66506 \\
    Western Washington University, Bellingham WA 98225 \\
    Bryn Mawr College, Bryn Mawr PA 19010 \\
    Brandeis University, Waltham MA 02454 \\[.5ex]} 

\begin{document}

\parskip 2pt

\maketitle

\vskip -40pt            

\begin{abstract}
The main result of this paper is that any $3$-dimensional manifold with a finite group action is equivariantly invertibly homology cobordant to a hyperbolic manifold; this result holds with suitable twisted coefficients as well.  The following two consequences motivated this work.  First, there are hyperbolic equivariant corks (as defined in previous work of the authors) for a wide class of finite groups. Second, any finite group that acts on a homology $3$-sphere also acts on a hyperbolic homology $3$-sphere. The theorem has other corollaries, including the existence of infinitely many hyperbolic homology spheres that support free $\bz_p$-\,actions that do not extend over any contractible manifolds, and (from the non-equivariant version of the theorem) infinitely many that bound homology balls but do not bound contractible manifolds.  In passing, it is shown that the invertible homology cobordism relation on $3$-manifolds is antisymmetric, and thus a partial order.
\end{abstract}

\renewcommand*{\thefootnote}{\fnsymbol{footnote}}
\newcommand{\foot}[1]{\setcounter{footnote}{1}\footnote{\ #1}}
            

\section{Introduction}\label{S:introduction}

The group of smooth homology cobordism classes of oriented 3-dimensional integral homology spheres has a complicated structure that is not fully understood, despite many advances coming from $4$-dimensional gauge theory; see for example \cite{fs:pseudofreeorbifolds,furuta:cobordism}.  This group appears in the theory of higher-dimensional manifolds, and also features prominently in the study of smooth $4$-manifolds.  The Rohlin invariant gives an epimorphism to $\bz_2$, and for a while this was all that was known about this group.  With the advent of gauge theory techniques \cite{donaldson:4manifolds}, it was shown to be infinite \cite{fs:pseudofreeorbifolds} (e.g.\ it is an easy consequence of Donaldson's diagonalization theorem \cite{donaldson:orientation} that the Poincare homology sphere represents an element of infinite order), indeed infinitely generated \cite{furuta:cobordism,FS:cob}. There have been many results since on the structure of this group, and on its applications, including Manolescu's spectacular resolution of the triangulation conjecture \cite{man}. 

It is interesting to explore how homology cobordism interacts with geometric structures on $3\text{-manifolds}$.  For example, there exist homology 3-spheres that are not homology cobordant to Seifert fibered homology spheres~\cite{stoffregen:seifert}, although the question of whether Seifert fibered spaces {\sl generate} the homology cobordism group is still unsolved.  In contrast, Myers~\cite{Myers2} proved that every $3$-manifold is homology cobordant to a hyperbolic manifold, and this result was later refined by Ruberman~\cite{ruberman:seifert} to show that such cobordisms can be taken to be invertible; the latter result has been applied to construct exotic smooth structures on contractible 4-manifolds \cite{akbulut-ruberman:absolute}.   

In this paper we show that any 3-manifold $M$ with a finite group action is equivariantly invertibly homology cobordant, with twisted coefficients, to a hyperbolic manifold.  The statement about twisted coefficients refines the earlier work of Myers and Ruberman, even in the non-equivariant setting: it shows that all the covering spaces of the constructed cobordism corresponding to subgroups of $\pi_1(M)$ are also homology cobordisms. As will be seen, this result has applications to 4-dimensional smooth topology and to the 3-dimensional space form problem, and may also be of interest in spectral geometry (cf.\,\cite{bartel-page:hyperbolic}).    

Throughout we work in the category of smooth, compact, oriented manifolds.  All group actions are assumed to be effective and orientation preserving, and all homology groups are assumed to have integer coefficients unless indicated otherwise.   

To state our main result, recall that a {\em homology cobordism} is a cobordism whose inclusions from the ends induce isomorphisms on homology.  For non-simply connected manifolds there is a stronger notion of {\it homology cobordism with twisted coefficients} in any module over the group ring of the fundamental group of the cobordism.  Also recall (see e.g.\ \cite{sumners:inv1,sumners:inv2,ruberman:seifert}) that a cobordism $P$ from $M$ to $N$ is {\em invertible} if there is a cobordism $Q$ from $N$ to $M$ with $P \cup_{N} Q \cong M\times I$. In this setting, there is a surjection from $\pi_1(P)$ to $\pi_1(M)$, and this surjection is implied when we talk about the homology of $P$ with coefficients in $\bz[\pi_1(\M)]$. See \secref{tech} for details, and the Appendix for a proof that invertible cobordism is a partial order on $3$-manifolds.

\begin{thm}\label{T:hypthm}
Any closed $3$-manifold $\M$ equipped with 
an action of a finite group $G$ is equivariantly invertibly $\bz[\pi_1(\M)]$-homology cobordant to a hyperbolic $3\text{-manifold}$ $N$ with 
a $G$-action by isometries.   This cobordism may be chosen to be a product along a neighborhood of the singular set of the action. 
 \end{thm}
 
We were led to this theorem by a question in 4-dimensional smooth topology.  Consider a finite group $G$ that acts on the boundary of some compact contractible $4\text{-dimensional}$ submanifold of $\br^4$, for example any finite subgroup of $SO(4)$.  In a recent paper~\cite{auckly-kim-melvin-ruberman:equivariant-cork} we constructed a compact contractible $4$-manifold $C$ with a $G\text{-action}$ on its boundary and an embedding of $C$ in a closed $4\text{-manifold}$ $X$ such that removing $C$ from $X$ and regluing by distinct elements of $G$ yields distinct smooth $4\text{-manifolds}$; related results were obtained by Tange~\cite{tange:cycliccorks} for $G$ finite cyclic and Gompf~\cite{gompf:infinite-cork} for $G$ infinite cyclic.   We call such a gadget a {\it $G$-cork}.  In our construction $\del C$ is reducible, and it is natural to ask if there exist $G\text{-corks}$ with irreducible or even hyperbolic boundaries; we call these {\it hyperbolic $G\text{-corks}$}.  Tange~\cite{tange:primecorks} has recently shown that his cyclic corks have irreducible boundaries, and (by computer calculations with HIKMOT~\cite{hikmot}) that some are hyperbolic. As a consequence of \thmref{hypthm}, proved in Sections 3--5, we will deduce: 

\begin{cor}\label{C:hypcork}
There exist hyperbolic $G$-corks for any finite group $G$ that acts on the boundary of some compact contractible $4$-dimensional submanifold of $\br^4$.
\end{cor}

The proof of this corollary will be given in \secref{corks}, along with the following applications to low dimensional topology. We start with a hyperbolic version of a non-extension result for group actions due to Anvari and Hambleton~\cite{anvari-hambleton:contractible}.

\begin{cor}\label{C:hyperextend}
For any Brieskorn homology sphere $\Sigma(a,b,c)$ and prime $p \nmid abc$, there is a hyperbolic homology sphere $N(a,b,c)$ with a free $\bz_p$-\,action that is $\bz_p$-\,equivariantly homology cobordant to the standard action on $\Sigma(a,b,c)$, and that does not extend smoothly  over any contractible $4$-manifold that $N(a,b,c)$ might bound. 
\end{cor}

We then apply \thmref{hypthm} in a non-equivariant setting to show that the difference between bounding an acyclic and contractible $4$-manifold occurs for hyperbolic homology spheres. 

\begin{cor}\label{C:acyclic}
There are infinitely many hyperbolic homology spheres that bound homology balls but do not bound contractible manifolds.
\end{cor}

The class {$\calg$} of all finite groups that act on homology 3-spheres (hyperbolic or not) include the finite subgroups of $SO(4)$ and some generalized quaternion groups of period 4 (as shown by Milgram~\cite{milgram:swan} and Madsen~\cite{madsen:spaceforms}; see also \cite[p.xi]{davis-milgram:survey}, \cite{kirby:problems96}).  It has been an open question since the early 1980s to determine exactly which groups lie in $\calg$, and to say something about the geometric nature of the homology spheres on which they act.  \thmref{hypthm} sheds light on this last question, especially for free actions.  

\begin{cor}\label{C:hyphomsphere}
Any finite group that acts on a homology $3$-sphere also acts on a hyperbolic homology $3$-sphere, with equivalent fixed-point behavior $($i.e.\ the two actions are equivariantly diffeomorphic near their fixed point sets$)$.  In particular, there exist finite groups that are not subgroups of $SO(4)$ $($so by geometrization do not act freely on $S^3)$ that act freely on some hyperbolic homology $3$-sphere. \end{cor}

\begin{remark*} Our results, in particular \thmref{hypthm} and \coref{hyphomsphere}, are related to a recent paper of Bartel and Page~\cite{bartel-page:hyperbolic} that constructs an action of an arbitrary finite group $G$ on a hyperbolic $3$-manifold $\M$ whose induced action on $H_1(\M;\bq)$ realizes any given finitely generated $\bq[G]$ module.  Indeed both our paper and theirs construct actions of finite groups on $3$-manifolds with prescribed homological action. However, neither paper implies the results of the other; for instance~\cite{bartel-page:hyperbolic} deals only with the action on rational homology, and does not provide a homology cobordism. On the other hand, our hyperbolization requires the existence of an action on some $3$-manifold as a starting point. It would be of interest to establish a sharper result realizing a given $\bz[G]$ module (even one with $\bz$ torsion) by an action on some $3$-manifold; our hyperbolization procedure would then produce such an action on a hyperbolic manifold. 

\coref{hyphomsphere} is also related to work of Cooper and Long~\cite{cooper-long:qhs} that constructs a free action of any finite group $G$ on a {\it hyperbolic} rational homology sphere.  It seems to have been known previously -- as pointed out to us by Jim Davis -- that the surgery-theoretic methods of~\cite{browder-hsiang:problems} give such an action on some (not necessarily hyperbolic) rational homology sphere; combining this observation  with \thmref{hypthm} thus gives an alternative proof of the Cooper-Long result.
\end{remark*}

In outline, the proof of \thmref{hypthm} is similar to the proofs of the analogous theorems in~\cite{Myers2} and \cite{ruberman:seifert}.  Start with a Heegaard splitting of $M$ of genus at least $2$ with gluing map $h$.  Then replace each handlebody, viewed as the exterior of a trivial tangle in the 3-ball, with the exterior of an invertibly null-concordant hyperbolic tangle.  To build the cobordism, glue the two concordances together by the map $h\times\id$. The top of the cobordism will be shown to be hyperbolic via the gluing techniques underlying Thurston's hyperbolization theorem~\cite{morgan,mcmullen:geometrization}.

To make this construction $G$-equivariant requires some modifications of this argument, even for free actions.  One starts with a Heegaard splitting of the orbit space $\M/G$, and then replaces the handlebodies with copies of the tangle exterior as in the outline above.  This gives an invertible cobordism from $\M/G$ to a hyperbolic $3\text{-orbifold}$.  Now if the action is free, then the induced $G$ cover is an invertible cobordism $P$ from $\M$ to a hyperbolic 3-manifold $\Mh$ that is equivariant with respect to the $G$-action.  However, there is no reason that $P$ should be a {\sl homology} cobordism, or indeed that  $\Mh$  should have the same homology as $\M$.  The issue is that while the tangle exteriors are $\bz$-homology equivalent to handlebodies, they are not necessarily homology equivalent with arbitrary (in this case $\bz[G]$) coefficients. This is of course familiar from knot theory; a covering space of a homology circle such as a knot complement need not be a homology circle.  The resolution of this issue is to further decompose each handlebody into 0 and 1-handles.  These handles will be replaced with {\it fake} $0$ and $1$-handles that will be hyperbolic tangle exteriors.  These are no longer homology handles, but rather homology handlebodies, but now one has control over their lifts.  

We will begin with some standard tangle exteriors, referred to as {\it atoms}, then glue these together by a bonding process to make the fake handles, and finally glue these fake handles together to make fake handlebodies and relative cobordisms. This localization will ensure that the replacement is homology cobordant to a handlebody $\H$ (with coefficients in $\bz[\pi_1(\H)]$) and again Thurston's gluing theorem will be used to create a closed hyperbolic manifold.  With some additional work, this argument extends to the case when $G$ has some fixed points. In this setting the quotient $\M/G$ will be an orbifold, and we will essentially be working with an orbifold Heegaard splitting.  

In our proof we will need tangles that are {\it simple} (a.k.a.\ {\it hyperbolic}) and {\it doubly-slice}, and that retain these properties as they are suitably glued together.  These notions will be made precise in \S2.  A basic example is the $4$-component boundary tangle $\R_4$ in the 3-ball displayed in \figref{tangles}a, with Seifert surface $\F_4$ shown in \figref{tangles}b.  Its $n$-component generalization $\R_n$, with Seifert surface $\F_n$, is the lift of the {\it generating arc} $\R \subset B^3$ shown in \figref{tangles}c to the $n$-fold branched cover along the {\sl axis} $\A$ perpendicular to  the page.  Note that the endpoints of $\R_n$ all lie on the equator $\calo_n$ of $B^3$ lying in the page, linking $\A$ once.  We will refer to the $\R_n$ for $n\ge2$ as {\it atomic tangles}, and to the $\calo_n$ as their {\it atomic orbits}. The atomic tangles were the key players in the last author's construction of invertible homology cobordisms in \cite{ruberman:seifert}.   
\begin{remark} \label{R:disjoint}
Since $\R$ and $\A$ are disjoint, so is each $\R_n$ and (the lift of) $\A$; this will ultimately be the reason why the cobordism in \thmref{hypthm} is a product near the singular set.  
\end{remark}

\fig{100}{FigTangles}
{\put(-380,-20){\small a) The tangle $\R_4$ and its orbit $\calo_4$ \hskip 15pt b) The Seifert surface $\F_4$ \hskip 25pt c) The generating arc $\R$}
\put(-54,5){\small$\R$}
\put(-54,54){\small$\A$}
\put(-192,54){\small$\A$}
\put(-323,54){\small${\A}$}
\put(-192,8){\small${\F_4}$}
\put(-325,8){\small${\R_4}$}
\put(-360,90){\small${\calo_4}$}
\caption{Atomic tangles in the $3$-ball}
\label{F:tangles}
\adjust[10]}

The following technical proposition, extracted from the proof of Theorem 2.10 in \cite{ruberman:seifert}, establishes the desired properties of the tangles $\R_n$.  
\begin{proposition}\label{P:rub}
The atomic tangles $\R_n\subset B^3$ are  {\,\small\bf a)} simple for all $n\ge 3$.\foot{In fact $\R_2$ is simple as well; see \remref{R2}.} and {\,\small\bf b)} doubly-slice for all $n$.
\end{proposition}

\noindent The proof is reviewed in the next section in the process of analyzing the more complicated tangles that arise in our constructions.
\begin{acknowledgement*}
The authors thank Ian Hambleton for input regarding the $3$-dimensional space form problem, and Jeff Meier for a helpful communication about the boundary of the Mazur manifold.
\end{acknowledgement*}

\section{Technical background}\label{S:tech}

\head{Marked Tangles}

In this paper, a {\it tangle} $T$ in a 3-manifold $M$ will refer to a finite disjoint union of proper arcs in $M$, called the {\it strands} of $T$.  Thus closed loops are not allowed.  We also assume that the endpoints of any given strand lie in the {\sl same} component of $\del M$.  A {\it marking} of $T$ is a disjoint collection $A$ of arcs in $\del M$ that join each strand's endpoints.  This gives a way to close $T$ into a link $T\cup A$ with the same number of components.  The pair $(T,A)$ is called a {\it marked tangle}, and $T\cup A$ its {\it associated link}.  Two marked tangles $(S,A)$ and $(T,B)$ in $M$ are considered {\it equivalent} if $(M,S\cup A)$ and $(M,T\cup B)$ are pairwise diffeomorphic.  

Tangle markings are not generally unique.   For example, two distinct markings $A_0$ and $A_1$ for the atomic tangle $\R_4 \subset B^3$ from \figref{tangles} are shown in \figref{markedtangles}.  In this picture, the part of the tangle inside the ball is drawn in muted tones.  In subsequent pictures we will omit these interior strands, but will always either draw the markings on the surface, or indicate in some other way where they lie.

\begin{remark}
The {\it canonical marking} $A_0$ is the unique marking of $\R_4$ that lies in its atomic orbit $\calo_4$.  Its position on the boundary of a $3$-ball $B$ in which $\R_4$ lies can be specified by drawing $\calo_4 \subset \del B$ marked with one point from each component of $A_0$.  The analogous marking of any atomic tangle $\R_n$, also denoted $A_0$  and called its {\it canonical marking}, can likewise be specified by drawing $\calo_n\subset\del B^3$ marked with $n$ points.

\fig{80}{FigMarkings}{
\put(-193,-20){\small a) $A_0$ (canonical)  \hskip 75pt b) $A_1$}
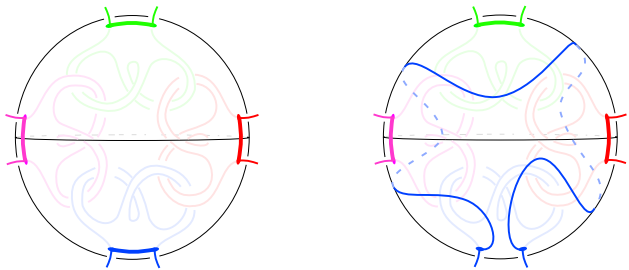
\caption{Two markings for the atomic tangle $\R_4$}
\label{F:markedtangles}
\adjust[10]}

At the end of this section, we will build some more complicated marked tangles called {\it $K$-molecules}  by assembling canonically marked atomic tangles $\R_n$ lying in copies of the $3$-ball.  Each $3$-ball will be viewed as a product $D^2\times D^1$ with axis $\A = \{0\}\times D^1$.  We will assume that the orbit $\calo_n$ is transverse to the circles $C_\pm = \del D^2\times\{\pm1\}$, and is marked with the points in $\calo_n\cap(C_+\cup C_-)$.  It follows that $n = |\calo\cap(C_+\cup C_-)|$ will always be even.
 \end{remark}  

\begin{defn}\label{D:index}
The {\it index} of any canonically marked atomic tangle $\R_n\subset D^2\times I$ as above is the pair of geometric intersection numbers $(n_+,n_-) = (|\calo_n\cap C_+|,|\calo_n\cap C_-|)$.  Thus $n_++n_- = n$, and each tangle strand runs from the cylinder $\del D^2\times D^1$ on the side to the top or bottom boundary disk, straddling $C_+$ or $C_-$.  
\end{defn}

\begin{remark}\label{R:lifts}
The lift of an $\R_n\subset D^2\times D^1$ of index $(n_+,n_-)$ to the $k$-fold cyclic cover of $D^2\times D^1$ branched along $\A$ is an $\R_{kn}\subset D^2\times D^1$ of index $(kn_+,kn_-)$.
\end{remark}

For our purposes, only tangles of index $(n,0)$ or $(n/2,n/2)$ will arise; their strands will either all straddle the top circle $C_+$, or half straddle $C_+$ and half straddle $C_-$.  See \figref{handletangles} for two representative examples.  

\fig{60}{FigHandles}{
\put(-325,-17){\small a) \ $\R_4$ of index $(4,0)$ \hskip 148pt b) \ $\R_8$ of index $(4,4)$}
\put(-48,25){$\A$}
\put(-280,25){$\A$}
\put(-90,27){$\calo_8$}
\put(-246,24){$\calo_4$}
\caption{Tangles in $D^2\times I$}
\label{F:handletangles}
\adjust[10]}

Several classes of marked tangles are relevant to our discussion, including (in increasing generality) {\it trivial, elementary}, {\it doubly-slice} and {\it boundary tangles}.  Here are the precise definitions of all but the doubly-slice class, postponed until later in this section:

\begin{defn}\label{D:elementary}
A marked tangle $(T,A)$ in a $3$-manifold $M$ is {\it trivial} if its associated link $T\cup A$ is an unlink.  More generally $(T,A)$ is called a {\it boundary tangle} if $T\cup A$ is a boundary link, meaning its components bound disjoint surfaces in $M$.  The union of these surfaces, positioned to meet $\del M$ in $A$, is then called a {\it Seifert surface} for $(T,A)$ with {\it inner boundary} $T$ and {\it outer boundary} $A$.
In particular, if $(T,A)$ has a Seifert surface $F$ with a {\it good basis}, meaning embedded curves $\alpha_1,\beta_1,\dots,\alpha_n,\beta_n$ representing a basis for $H_1(F)$, disjoint except for a single point in $\alpha_i \cap\beta_i$ for each $i$, and satisfying the two properties
\items
\item[{\small\bf a)}] the $\alpha_i$ bound disjoint disks in $M$ whose interiors intersect $F$ only in arcs transverse to the $\beta$ curves 
\item[{\small\bf b)}] the $\beta_i$ bound disjoint disks in $M$ whose interiors intersect $F$ only in arcs transverse to the $\alpha$ curves
\stopitems
then $(T,A)$ is called an {\it elementary tangle}.   
\end{defn}

The markings in this definition are essential.\foot{Classically, an unmarked tangle $T$ is called a boundary tangle if $(T,A)$ is a boundary tangle in the sense of \defref{elementary} for {\sl some} marking $A$ of $T$.  Not all tangles are boundary tangles in this classical sense; for example the invariants $I^n$ from~\cite{cochran-ruberman:tangles} give obstructions, and these are all realized using~\cite{cochran:gilc}.}  For example $(\R_4,A_1)$ is not a boundary link, as its linking matrix is nonzero, while $(\R_4,A_0)$ and more generally all the canonically marked atomic tangles $(\R_n,A_0)$ are; they have the obvious Seifert surfaces $\F_n$ of genus one components, illustrated for $n=4$ in \figref{tangles}b.   In fact letting $\alpha_i$ (resp.\ $\beta_i$) be the obvious closed curves traversing the bands on the left (resp.\ right) side of each component of these surfaces -- when viewed with its outer boundary at the bottom -- we see that:

\begin{lemma} \label{L:elementary}
All canonically marked atomic tangles are elementary.
\end{lemma} 

\head{Tangle Sums}

Tangles can be added together in a variety of ways.  For the present purposes, the following notions of {\it tangle sums} and {\it marked tangle sums} will suffice: 

\begin{defn}\label{D:tanglesum} 
Given tangles $T_i\subset M_i$ for $i=1,2$, choose {\it gluing disks} $D_i\subset \del M_i$ containing an equal number of tangle endpoints in their interiors.  Then glue $M_2$ to $M_1$ by a diffeomorphism $h\colon D_2\to D_1$ that identifies these endpoints without creating loops.  The result is the {\it tangle sum} $T_1 +_h T_2 := T_1\cup_h T_2 $ in the boundary connected sum $M = M_1\cup_h M_2$.  The common image $D\subset M$ of the $D_i$ after gluing is called the {\it splitting disk} of the sum.   More generally we allow the $D_i$ (so also $D$) to be unions of more than one disk.
\end{defn}

To propagate hyperbolic structures on tangles to their tangle sums -- see \propref{gluing} below -- we will use a restricted class of {\it simple tangle sums}:

\begin{defn}\label{D:simplesum}
A tangle sum $T_0+_hT_1$ as in \defref{tanglesum} is {\it simple} if each gluing disk in $D_i$ contains at least two tangle endpoints, and each component of $\del M_i-D_i$ contains at least two tangle endpoints if it is a disk and three if it is a sphere.  In other words 
each component of $D_i-T_i$ and $(\del M_i-D_i)-T_i$ has negative euler characteristic, or equivalently, is not a sphere, torus, disk or annulus,
\end{defn}

\begin{defn}\label{D:markedtanglesum} If the tangles $T_i\subset M_i$ in \defref{tanglesum} have markings $A_i \subset \del M_i$ that {\it straddle} the boundaries of the gluing disks $D_i$ (meaning each arc component of $A_i$ that intersects $D_i$ meets $\del D_i$ transversely in a single point), and if $h$ {\it preserves} the markings (meaning $h(A_2\cap D_2) = A_1\cap D_1$), then $T_1 +_h T_2$ acquires a natural marking $A_1+_h A_2 := (A_1\cup_h A_2)\cap\del M$, and the result is the {\it marked tangle sum} $(T_1,A_1)+_h(T_2,A_2) :=(T_1 +_h T_2,A_1+_h A_2)$. 
\end{defn} 

A marked tangle sum is shown in \figref{tanglesum}.  It is immediate from the definition -- by gluing the relevant Seifert surfaces together -- that marked sums of boundary tangles are again boundary tangles.  The same is true for elementary tangles and {\it doubly-slice tangles}  (defined below; see \propref{ds}a).

\fig{90}{FigSum}{
\put(-440,20){\small$M_0$}
\put(-202,20){\small$M_1$}
\put(-383,34){\small$D_0$}
\put(-310,34){\small$D_1$}
\put(-347,4){\small$h$}
\put(-50,7){\small$M_0\cup_h M_1$}
\caption{A marked tangle sum $T_0\cup_hT_1 \ \subset \ M = M_0\cup_h M_1$}
\label{F:tanglesum}
\adjust[10]}

\head{Thurston's hyperbolization and simple tangles}\label{S:hyperbolic}

The 3-manifolds that we construct in the course of proving \thmref{hypthm} will be shown to be hyperbolic using Thurston's hyperbolization theorem for Haken 3-manifolds \cite{thurston:hyp,kapovich:hyperbolic,otal:fibred,otal:haken} and standard techniques for checking that 3-manifolds obtained by gluing satisfy the hypotheses of his theorem.

To efficiently state Thurston's theorem and the relevant gluing results, it is convenient to call a surface in a $3$-manifold {\it essential} if it is compact, properly embedded, incompressible and
nonboundary-parallel (see \cite{waldhausen} for the definitions).  We also assume implicitly unless indicated otherwise that all $3$-manifolds are compact, oriented, irreducible (no essential spheres) and boundary irreducible (incompressible boundary).  
Such a $3$-manifold is said to be {\it Haken} if it contains an essential surface, and {\it simple} if it contains no essential tori, annuli or disks.
The analogous conditions for tangles will also be used:

\begin{defn}\label{D:simpletangle}
A tangle in a $3$-manifold is {\it Haken} if its exterior (the complement of an open tubular neighborhood) is Haken, and {\it simple} if its exterior is simple.
\end{defn}

Thurston's theorem for {\it closed} $3\textup{-manifolds}$ asserts that these conditions together -- the presence of an essential surface of Euler characteristic $<0$ but none of Euler characteristic $\ge0$ -- imply that the manifold is either {\it small} (meaning Seifert fibered over the $2$-sphere with three singular fibers) or hyperbolic. 

\begin{theorem}
\hskip-3pt{\rm(Thurston)}\ Any closed simple Haken $3$-manifold that is not a small Seifert fibered manifold admits a complete hyperbolic metric.
\end{theorem}

The hyperbolization result of the last author~\cite{ruberman:seifert} relied on proving that the atomic tangles $\R_n$ from \S1 are simple for $n\ge3$ (\propref{rub}a).  We also need this result here, but will not repeat the proof.\foot{Here is a sketch from \cite{ruberman:seifert}:  Recall that $\R_n$ is the lift of an arc $\R$ in the $3$-ball under a branched cover along an axis $\A$ of the ball  disjoint from $\R$.  Let $X = B^3-\inte(R\cup D)$, where $R$ and $D$ are closed tubular neighborhoods of $\R$ and $\A$, and let $P$ be the annulus $\del X \cap R$.   Gluing results from \cite{morgan} can be used to show that $X$ is hyperbolic in a certain sense, in particular the pair $(X,P)$ is a `pared manifold'.  The result then follows by standard arguments about incompressible surfaces in branched covers, using equivariant versions of the loop, sphere, annulus and torus theorems (see Theorem 2.10 in \cite{ruberman:seifert}).  A direct proof that $\R_n$ is simple for $n\ge3$ using \lemref{surf} would also be desirable; see \remref{R2} for such a proof when $n=2$.}  In addition, we need a way to see that certain manifolds built from these atoms are simple. The necessary gluing results can be found in Myers' work \cite{Myers1,Myers2}.

\begin{defn}\label{D:simple}
Let $\M$ be a compact irreducible 
$3$-manifold and $F$ a compact subsurface of $\del \M$.  A properly embedded surface $S\subset M$ is of {\it type k} (with respect to $F$) if $\del S$ is transverse to $\del F$ in $\del M$, and $S\cap F$ consists of $k$ arcs, all essential in $F$, and any number of circles.  The pair $(M,F)$ is \emph{simple} if it satisfies the properties
\items
\item[{\small\bf a)}] $F$ contains no torus, annulus or disk components, and 
\item[{\small\bf b)}] $\M$ contains no essential tori, annuli of type $0$, or disks of type $\le1$,
\stopitems
\noindent and {\it very simple} if it also satisfies {\small\bf c)} $M$ contains no essential disks of type $2$.  
\end{defn}

\begin{remarks}\label{R:simple}
{\small\bf a)} \ Arcs in $\del S\cap F$ that are inessential in $F$ can always be removed by an isotopy of $S$.  Thus every properly embedded surface in $M$ is isotopic to one of type $k$ for some $k$.  \\[2pt]
\noindent {\small\bf b)} \ Having no essential disks of type $0$ is equivalent to the incompressibility of $F$ and $\del M - F$ in $M$.  \\[2pt] 
\noindent {\small\bf c)} \ When the pair $(M,F)$ is simple, $M$ itself need not be simple.  In particular, $\del M$ may be compressible in $M$ and/or $M$ may contain incompressible annuli or disks, but the annuli of type 0 and disks of type $\le1$ (and also type $2$ for very simple pairs) must all be boundary parallel.  However, {\it If $M$ is simple, then the pair $(M,F)$ is very simple if and only $F$ satisfies just property {\small\bf a)} in} \defref{simple}; properties {\small\bf b)} and {\small\bf c)} are automatic since $M$ contains no essential tori, annuli or disks whatsoever.  \\[2pt]  
\noindent {\small\bf d)} \ A compact irreducible $3$-manifold $M$ {\sl is} simple if and only if the pair $(M,\varnothing)$ is (very) simple. 
\end{remarks}

These notions of {\it simple} and {\it very simple} pairs are exactly Myers' {\it Properties} $B^\prime$ and $C^\prime$ that feature in the following gluing result  \cite[Lemma 2.5]{Myers2}, proved in Section 3 of \cite{Myers1}:

\begin{lemma}\label{L:surf}
\hskip-3pt{\rm(Myers)}\ If $(M_0,F_0)$ is very simple, $(M_1,F_1)$ is simple, and $h\colon F_1\to F_0$ is a homeomorphism, then $M= M_0\cup_h M_1$ is simple and Haken.
\end{lemma}

\noindent {\it Proof sketch}.   A standard innermost curve argument shows that the {\it splitting surface} $F$ in $M$ (the common image of $F_0$ and $F_1$ after the gluing) is incompressible, so $M$ is Haken.  If $S$ is any proper disk, incompressible annulus or incompressible torus in $M$, then innermost curve and outermost arc arguments using conditions {\small\bf b)} and {\small\bf c)} in \defref{simple} show how to isotop $S$ off of $F$, and thus into $M_i$ for some $i\in\{0,1\}$.  It follows from the simplicity of $(M_i,F_i)$ that $S$ is parallel to a surface in $\del M_i$, which by condition {\small\bf a)} actually lies in $\del M_i-F_i\subset \del M$.  Thus $S$ is inessential in $M$, proving that $M$ is simple.    \qed

\begin{remark} \label{R:R2} The conclusion of the lemma still holds under the weaker assumption that $h$ is an embedding of $F_1$ onto a union of components of $F_0$, essentially by the same argument.  From this one can give a quick proof that {\it the atomic tangle $\R_2\subset B^3$ is simple} (a result of independent interest, not needed in this paper):  First observe that $\R_2$ is isotopic to the {\it pretzel tangle} $T_{3,-3,3}$ shown in \figref{R2}a.  Viewing any tangle $T$ as lying in $I\times D^2$, where both strands enter through the left disk $T^-=\{0\}\times D^2$ and exit through the right disk $T^+=\{1\}\times D^2$, we see that $T_{3,-3,3} = T_{3,-3}+_hT_3$, where $h\colon T_3^-\to T_{3,-3}^+$, as indicated in \figref{R2}b.  Now appeal to the fact that $(T_{3,3},T_{3,3}^-\cup T_{3,-3}^+)$ is very simple and $(T_3,T_3^-)$ is simple, the latter proved in \cite[\S4]{Myers1} and the former by similar methods (exploiting the fact that there are only six proper arcs up to isotopy in a pair of pants).  
\end{remark}

\fig{70}{FigR2}{
\put(-370,-17){\small a) \ $R_2$ pictured as the pretzel tangle $T_{3,-3,3}$ \hskip 47pt b) \ Decomposition as a tangle sum $T_{3,-3}+_hT_3$}
\put(-375,35){$T_{3,-3,3}^-$}
\put(-225,35){$T_{3,-3,3}^+$}
\put(-140,35){$T_{3,-3}^-$}
\put(2,35){$T_{3}^+$}
\put(-68,75){$T_{3,-3}^+ \!\overset{h}{\from}\!T_3^-$}
\caption{Decomposition of the atomic tangle $\R_2$}
\label{F:R2}
\adjust[10]}

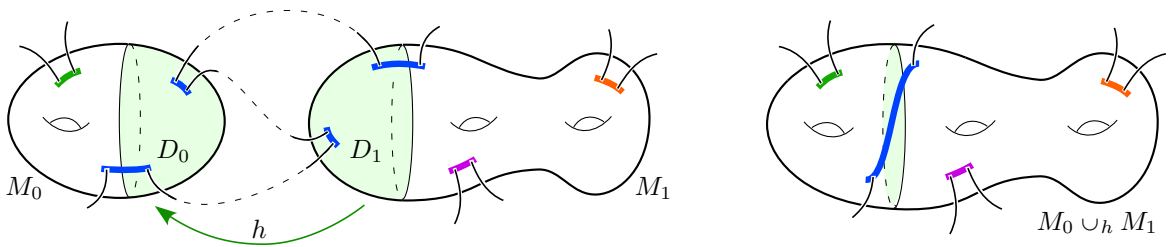
We apply these gluing techniques in two situations, the first when the gluing surfaces have boundary, and the second when they are closed.  For the bounded case, consider a {\sl simple} tangle sum $T_0 \cup_h T_1$ of {\sl simple} tangles $T_i\subset M_i$ with gluing disks $D_i\subset \del M_i$.  Write $X_i$ for the exterior of $T_i$ in $M_i$, and set $Y_i = D_i\cap X_i$.  By \defref{simplesum}, this means that no components  of $Y_i$ or $\del X_i-Y_i$ are tori, annuli or disks.  Thus both pairs $(X_i,Y_i)$ are very simple by \remref{simple}c, so the exterior $X_0\cup_h X_1$ of $T_0 \cup_h T_1$ in $M_0\cup_h M_1$ is simple (and Haken) by \lemref{surf}.  This proves the first part of the following result; the second part is an immediate consequence of \remref{simple}c, \lemref{surf} and Thurston's theorem.  

\begin{proposition}\label{P:gluing}{\small\bf a)} \ Any simple sum of simple tangles is a simple tangle. \\[2pt]
\noindent {\small\bf b)} \ Any closed 3-manifold obtained by gluing together a pair of simple 3-manifolds is hyperbolic or a small
Seifert fibered space.  \qed
\end{proposition}

Our proof of \thmref{hypthm} will rely on this proposition in the following way:  Starting with a Heegaard splitting $\H_0\cup_h\H_1$ of a 3-manifold $\M$, we will repeatedly apply \propref{gluing}a (and \propref{ds}a below) to construct simple, doubly-slice (as defined in the next subsection) {\it molecular tangles} in $\H_0$ and $\H_1$, with an equal number of strands.  Then gluing their exteriors $\rH_0$ and $\rH_1$ together by a natural map $h_\text{hyp}$ induced by $h$ will yield a simple manifold $\Mh$, by \propref{gluing}b.  A
small modification will then show that the result is not a Seifert Fibered space. If the Heegaard splitting of $\M$ is equivariant with respect to the action of a finite group $G$ on $\M$ (in a strong sense explained in the next section), and the simple tangles in the $\H_i$ are suitably chosen, then $\M$ and $\Mh$ will be invertibly homology cobordant (as defined below). The details of this construction will be explained in the next three sections. 

To complete the proof of \thmref{hypthm} we will need to show that the orbifold $\M/G$ is hyperbolic. There are two approaches: either expand the discussion above to include a definition of simple orbifold pairs, and argue that the gluing results hold in this more general setting, or make use of Thurston's orbifold theorem~\cite{boileau-leeb-porti:orbifolds,cooper-hodgson-kerckhoff:orbifolds}.  We follow the latter approach, and in fact need only the following special case:
\begin{theorem}\label{T:orbifold}
Any action of a finite group $G$ on a closed hyperbolic $3$-manifold $\M$ is conjugate to an action by isometries, and so $\M/G$ is a hyperbolic orbifold.
\end{theorem}
\noindent This is due to Wang~\cite{wang:action} if the action has fixed points (meaning points with nontrivial stabilizers), and follows by geometrization \cite{perelman:ricci,perelman:extinction, perelman:surgery, morgan-tian:geometrization} or by~\cite{gabai-meyerhoff-thurston} for free actions.

\head{Invertible cobordisms and doubly-slice tangles}

Invertible cobordisms of manifolds, knots and links have been studied since the 1960s; see e.g.\ \cite{fox:trip,sumners:inv1,sumners:inv2,gordon-sumners:ball-pairs}.  For manifolds $M$ and $N$ of the same dimension whose boundaries (if nonempty) are identified by a diffeomorphism $h$, a {\it cobordism} from $M$ to $N$ is a manifold $P$ with $\del P = -M\cup_h N$ (\,$=-M\sqcup N$ when the boundaries are empty).  Thus in the bounded case $P$ can be viewed as a relative cobordism from $M$ to $N$ with the vertical part of $\del P$ diffeomorphic to $\del M\times I$, extending $h$ at the top.  This cobordism $P$ is said to be {\it invertible} if there is a cobordism $Q$ from $N$ to $M$ such that $P\cup_N Q \cong \M\times I$.  We then say that $\M\times I$ is {\it split} along $N$, and call $Q$ {\it an inverse} of $P$.  
Familiar examples of 3-dimensional invertible cobordisms arise from any homology 3-sphere $\M$ that bounds a contractible 4-manifold whose double is the 4-sphere; the complement in the $4$-manifold of an open 4-ball is then an invertible cobordism from the 3-sphere to $\M$.  An important example of this type is given by the Mazur manifold $W$~\cite{akbulut-kirby:mazur,mazur:contractible}.

Similar language applies to {\it link concordances}, from one link $S$ to another $T$ in a manifold $M$, meaning embeddings of disjoint annuli in $M\times I$ that stretch from $S$ to $T$.  Such a concordance is {\it invertible} if it can be followed by a concordance from $T$ to $S$ to produce the product concordance from $S$ to itself.  If $S$ is an unlink, then $T$ is said to be {\it invertibly null-concordant} or {\it doubly-slice}.  

\begin{remark} \label{R:invertible} The relations of invertible cobordism and concordance are clearly reflexive and transitive, but generally not symmetric.  
For example, any sphere is invertibly cobordant to a disjoint union of two spheres, but not conversely, and analogously an unknot is invertibly concordant to a two component unlink, but not conversely.  In fact, for closed manifolds of dimension 3 or less, invertible cobordism is an {\it antisymmetric} relation, and thus a partial order.  For hyperbolic $3\text{-manifolds}$ this follows from degree and volume considerations (cf.\  \cite[Theorem C.5.5]{bp}) and a general proof for 3-manifolds is given in the appendix, where it is also noted that antisymmetry fails in higher dimensions. 
\end{remark}

These notions have also been studied for {\it tangle concordances} in $3$-manifolds, where there is the added requirement that the concordance must be a product along the boundary (see e.g.\ \cite{ruberman:seifert} \cite{kim}).  It follows that the tangles at the ends will have the same endpoints, and perhaps more significantly, will have exteriors whose boundaries are naturally identified; this result will be used in the proof of \thmref{hypthm} in \secref{replacement}:

\begin{lemma}\label{L:longitude}  
Let $C\subset M\times I$ be a concordance between tangles $T_0$ and $T_1$ in a $3$-manifold $M$.  Choose open tubular neighborhoods $N_i$ of $T_i\subset M$ that agree on $\del M$, and set $E = N_i\cap\del M$.  Then there is a canonical identification between $\del(M-N_0)$ and $\del(M-N_1)$ extending the identity on $\del M-E$.  \end{lemma}

\begin{proof}
It suffices to show that a choice of longitudes for the strands of $T_0$ canonically induces a choice of longitudes for the strands of $T_1$.  
Since $C$ is topologically a union of rectangles, it has a trivial normal bundle in $M\times I$.  The longitudes for $T_0$ correspond to a trivialization of this bundle along one side of each rectangle.  These extend trivially along the adjacent sides of the rectangles (lying in $\del M\times I$) and then uniquely across the rest of $C$, restricting to the desired trivializations along $T_1$. 
\end{proof} 

When considering concordances between {\sl marked} tangles, the product structure along the boundary allows us to compare the markings at the ends of the concordance, and we require these to be the {\sl same}\,: marked tangles $(T,A)$ and $(T',A')$ are {\it concordant} if and only if $T$ and $T'$ are concordant and $A=A'$.  We then have the following notions for marked tangles, analogous to the corresponding notions for links:

\begin{defn} \label{D:doublyslice}
A concordance from one marked tangle $(S,A)$ to another $(T,A)$ in a $3$-manifold $M$ is {\it invertible} if it can be followed by a concordance from $(T,A)$ to $(S,A)$ to produce the product concordance from $(S,A)$ to itself.  If $(S,A)$ is trivial, then $(T,A)$ is said to be a {\it doubly-slice tangle}.\foot{Note that all doubly-slice tangles $(T,A)$ are boundary tangles; a Seifert surface is obtained by intersecting the union of $3$-balls bounded by the union of an invertible concordance from a trivial tangle to $(T,A)$ and its inverse with the middle level between the two.  The converse is false, e.g.\ a knotted trefoil arc is a boundary tangle but not doubly-slice (nor even slice).}  
\end{defn}


The following proposition gives tangle versions of standard properties of invertible concordances of knots and links. The first part generalizes the fact that connected sums of doubly-slice knots are doubly-slice, while the second is a relative version of a well known double slicing technique introduced by Terasaka and Hosokawa~\cite{terasaka-hosokawa} (cf.\ \cite[Proof of 2.6]{ruberman:seifert}). 

\begin{proposition}\label{P:ds}
{\small\bf a)} \ Marked tangle sums of doubly-slice tangles are doubly-slice. \\[2pt]
\noindent{\small\bf b)} \ Elementary tangles are doubly-slice, and thus all canonically marked atomic tangles $\R_n$ are doubly-slice \ {\rm(by \lemref{elementary}, proving \propref{rub}b)}.
\end{proposition}

\begin{proof}
For {\small\bf a)} let $(T_1,A_1)+_h(T_2,A_2) \subset M_1\cup_h M_2$ be a marked tangle sum, with gluing $h\colon D_2\to D_1$ as in \defref{markedtanglesum}.  We can define the {\it sum} of any pair of concordances between $(T_i,A_i)$ and another marked tangle $(S_i,A_i)$ (for $i=1,2$) by gluing them together using $h\times\id\colon D_2\times I\to D_1\times I$.  Since the sum of product concordances is evidently a product, it follows that the sum of invertible concordances is invertible.  Thus if the $(T_i,A_i)$ are doubly-slice, then so is $(T_1,A_1)+_h(T_2,A_2)$.

To prove part {\small\bf b)} let $(T,A)$ be an $n$-stranded elementary tangle in $M$, and $F$ be a corresponding Seifert surface with  good basis $\{\alpha_i,\beta_i\}$ as in \defref{elementary}.  View $F$ as $n$ disjoint disks with bands attached along the basis curves.  Removing the $\beta$ bands from $F$ yields a surface $F_0\subset M$ that can be capped off with (parallel copies of) the disks bounded by the $\alpha$ curves
to form a trivial Seifert surface $E_0\subset M$ for a trivial tangle $(U_0,A)$.  Similarly form $F_1\subset M$ by removing the $\alpha$ bands from $F$, and then cap off with disks bounded by the $\beta$ curves to produce another trivial Seifert surface $E_1\subset M$ for a trivial tangle $(U_1,A)$.  By construction, $E_0$ and $E_1$ have the same outer boundary $A$ as $F$. 

Now build a 3-dimensional cobordism $\bp \subset M\times[0,1/2]$ 
from $E_0\subset M\times0$ to $F\subset M\times 1/2$, with outer lateral boundary $A\times[0,1/2]$ and inner lateral boundary a concordance $P$ from $(U_0,A)$ to $(T,A)$, as follows:  Start with $\bp$ as $F\times [0,1/2]$ with $2$-handles attached ambiently in $M\times[-1/2,0]$ along the $\alpha$ curves in $F\times0$.  Then push $\bp$ up  from its bottom level $E_0$ so that it lies in $M\times[0,1/2]$.  (A top down movie of the inner lateral boundary $P$ of $\bp$ 
is described as follows:  Start with $T$.  Then perform saddle moves along the cocores of the $\beta$ bands, tracing out a genus zero cobordism from $T$ to $\del F_0$.  Finish by capping off the $\alpha$ curves with disjoint disks.). Similarly build a cobordism $\bq \subset M\times[1/2,1]$ from $F\subset M\times 1/2$ to $E_1\subset M\times1$ with outer lateral boundary $A\times[1/2,1]$ and inner lateral boundary a concordance $Q$ in $M\times[1/2,1]$ from $(T,A)$ to $(U_1,A)$.  

Then $\bp\cup\bq$ is a product cobordism.  Indeed, since $|\alpha_i\cap\beta_j| = \delta_{ij}$, the 1-handles (upside down 2-handles) in $\bp$ are cancelled by the 2-handles in $\bq$, so $\bp\cup\bq$ is in fact a union of 3-balls.  It follows that $P$ is the desired null-concordance of $(T,A)$, with inverse $Q$. 
\end{proof}

\head{Homology cobordisms} \label{S:invertible}

A {\it homology cobordism} is a cobordism for which the inclusions from the ends induce isomorphisms on homology with integer coefficients (understood throughout unless stated otherwise).  It is a standard and very useful observation that a concordance between knots or links induces a homology cobordism  between their complements; see for instance~\cite{gordon:contractible}.  We note a somewhat stronger property for the concordances constructed in the previous subsection.  Let $X$ be the exterior of a tangle $T$ in a $3$-manifold $\M$.  The inclusion of $X\hookrightarrow\M$ induces a homomorphism $\pi_1(X) \to \pi_1(\M)$.  Thus any module $V$ over $\bz[\pi_1(\M)]$ is also a module over $\bz[\pi_1(X)]$, so we can consider the twisted homology $H_*(X; V)$. 

\begin{lemma}\label{L:H-cob}
Let $T_0$ and $T_1$ be tangles in a compact $3$-manifold $\M$, with exteriors $X_0$ and $X_1$, and $C$ be an invertible concordance in $M\times I$ from $T_0$ to $T_1$, with exterior $X$.  Then $X$ is an invertible homology cobordism from $X_0$ to $X_1$, with twisted coefficients in any $\bz[\pi_1(M)]$-module $V$.  
\end{lemma}

\begin{proof}
Since $C$ is invertible, so is $X$, with inverse the exterior of the inverse tangle concordance for $C$.  That $X$ is a homology cobordism with twisted coefficients (even when $C$ is not invertible) is implicit in~\cite{cappell-shaneson:surgery}; here is a quick proof for the reader's convenience:   By hypothesis $X = \M\times I - \inte N$, where $N$ is a tubular neighborhood of $C$, and $X_i = X\cap(M\times i)$.  It suffices to show $H_*(X,X_i;V) = 0$.  Set $N_i = N\cap(M\times i)$, and note that the restriction of the coefficient system $V$ to $\partial N$ (and similarly for the $\partial N_i$) is trivial, because it extends over $N$.  Then for $i=0$ and $1$, there are relative Mayer-Vietoris sequences
$$
\cdots \to H_*(X\cap N,X_i\cap N_i;V) \to H_*(X,X_i;V)\oplus H_*(N,N_i;V) \to H_*(M\times I,M_i;V) \to \cdots
$$
in which all the groups except $H_*(X,X_i;V)$ clearly vanish.  Thus $H_*(X,X_i;V) = 0$ as well. 
\end{proof}      

Our main theorem constructs, for any $3$-manifold $M$, an invertible homology cobordism from $M$ to a hyperbolic $3$-manifold. When $M$ has a $G$-action the cobordism will be equivariant in the sense that it has a $G$-action extending the one on $M$.  We give a preliminary result  in this direction that will be used in the proof of the main theorem, based on the invertible homology cobordism coming from the Mazur manifold $W$ referenced above.  The boundary $\del W$ is $+1$ framed surgery on the $(-3,3,-3)$ pretzel knot. Wu showed that $\del W$ is irreducible and atoroidal in~\cite{wu:arborescent,wu:montesinos} and Meier proved that it is not Seifert fibered in~\cite{meier:small}.  It follows from geometrization~\cite{perelman:ricci,perelman:extinction,perelman:surgery,morgan-tian:geometrization} that $\del W$ is hyperbolic, and so in particular has non-trivial Gromov norm~\cite{gromov:norm}. This gives an example of a $\bz[\pi_1]$-invertible homology cobordism from $S^3$ ($\pi_1 = 1$) to a hyperbolic manifold. This in turn gives an easy way to construct an equivariant $\bz[\pi_1(M)]$-invertible homology cobordism from any $3$-manifold $M$ with a $G$ action to one with non-trivial Gromov norm. 

\begin{lemma}\label{L:gn}
Any  $3$-manifold $M$ with a $G$-action is equivariantly invertibly $\bz[\pi_1(M)]$-homology cobordant to a $3$-manifold with non-zero Gromov-norm.
\end{lemma}

\begin{proof} Extend the $G$-action to $I\times M$ as the action on the second factor, and define a $G$ action on $G\times  (W\setminus B^4)$ as the action on the first factor.  In this case, one directly defines the invertible homology cobordism as the sum along submanifolds given by
\[
P:= (I\times M) \cup_{G(I\times x)} \perp\!\!\!\perp G\times  (W\setminus B^4)
\]
where $x$ is any point with trivial stabilizer in $M$ and $W$ is the Mazur manifold. Here we pick any embedding $I=[0,1]\to (W\setminus B^4)$ taking $0$ into $S^3$ and $1$ into $\partial W$. A sum along submanifolds requires that the submanifolds be framed. The framing here is not important as long as it is equivariant and that may be accomplished by picking a framing on one component and translating it by the group action. The inverse homology cobordism is obtained by reversing the orientation on $P$.
\end{proof}

\head{$K$-molecules}

It was seen above that all canonically marked atomic tangles $\R_n\subset B^3$ are simple and doubly-slice.  In the proof of \thmref{hypthm}, we will construct some more complicated {\it molecular tangles} in $3$-dimensional handlebodies, each built as a simple marked sum of {\sl many} atomic tangles.  These molecular tangles will again be simple by \propref{gluing}a and doubly-slice by \propref{ds}a (or alternatively, noting that a marked sum of elementary tangles is again elementary, by \propref{ds}b).  

The first step in this construction is to build what we will call {\it $K$-molecules}, a family of simple elementary tangles in the $3$-ball, roughly parametrized by $1$-complexes $K$ lying on the ball's boundary $2$-sphere:   

\begin{defn}\label{D:kmolecule}  
For any $1$-complex $K$ (with at least one edge) embedded in $\del B^3$, a {\it $K$-molecule}  $\R_K \subset B^3$ is a marked tangle obtained by assembling canonically marked atomic tangles along $K$ in the following way:  Expand $K$ into a $2$-dimensional handlebody $H \subset \del B^3$ with a $0$-handle about each vertex and a $1$-handle along each edge, and thicken this into a $3$-dimensional handle structure for $H\times D^1$ in a boundary collar $C$ of $B^3$, where $H\times\{1\}\subset \del B^3$.  Now (referring to \defref{index}) insert a canonically marked atomic tangle $\R_{2k}$ of index $(2k,0)$ in each thickened $0$-handle $D^2\times D^1$ for a vertex of degree $k$ in $K$, and an $\R_4$ of index $(4,0)$ in each thickened $1$-handle $D^1\times D^1\times D^1\cong D^2\times D^1$ (see \figref{handletangles}a).  In particular, insert these atomic tangles so that the markings on the thickened handles match up where they intersect, with {\sl exactly two marking arcs} meeting each component disk in the attaching region $\del D^1\times D^1\times D^1$.  This construction is illustrated in \figref{molecule}, with the markings shown in bright colors in c), and subterranean details in d).
\end{defn}

\fig{90}{FigMolecule}{
\put(-415,-17){\small a) \ $1$-complex $K\subset\del B^3$ \hskip 20pt b) \ Handlebody $H\subset\del B^3$ \hskip 25pt c) \ $K$-molecule
\hskip 37pt \ d) Details}
\caption{$K$-molecule}
\label{F:molecule}
\adjust[10] }

\begin{remarks}\label{R:kmolecule}
{\small\bf a)} \ The insertions of atomic tangles in the handles of $H\times I$ are not unique.  Thus the notation $\R_K$ does not specify a particular $K$-molecule, but any one will do for our purposes.  \\[2pt]
{\small\bf b)} \  All $K$-molecules are boundary tangles with genus two Seifert surface components, each arising as a boundary sum of genus one Seifert surface components from a pair of atomic tangles (see \figref{molecule}d). \\[2pt]
{\small\bf c)} \  If $T\subset B^3$ is a $K$-molecule, and $p\colon S^2\to S^2$ is a cover branched along the vertices $V$ of $K$ with induced cover $P\colon B^3\to B^3$ branched along the cone on $V$, then $P^{-1}(T)\subset B^3$ is a $p^{-1}(K)$-molecule.  This follows readily from \remref{lifts}, and is a key ingredient in our proof of equivariance in \thmref{hypthm}.   
\end{remarks}

\begin{lemma}\label{L:kmolecule}
Every $K$-molecule $\R_K\subset B^3$ is a $4e$-stranded simple, doubly-slice $($indeed elementary$)$ marked tangle, where $e$ is the number of edges in $K$.
\end{lemma} 

\begin{proof} By construction, $\R_K$ is a marked tangle sum of atomic tangles.  We organize this sum as follows.  Choose an ordering $h_1,\dots,h_n$ of all the handles in $H\times D^1$ (thus $n=v+e$ if $K$ has $v$ vertices) so that the union $h_1\cup\cdots\cup h_k$ is connected for each $k\le n$.  For convenience we assume $h_1$ is a $1$-handle.  Let $R_i$ be the atomic tangle inserted in $h_i$ as in \defref{kmolecule}, and $B$ be the $3$-ball (with corners) that is the closure of the complement of $H\times D^1$ in $B^3$.  Then $B^3$ is  a nested union of $3$-balls $B_1\subset B_2\subset\cdots\subset B_n = B^3$, where $B_i = B\cup h_1\cup\cdots\cup h_i$ contains the marked tangle $T_i = R_1\cup\cdots\cup R_i$.  In particular $T_n=\R_K \subset B^3$.

Now observe that $T_1=R_1\subset B_1$ is a copy of $\R_4\subset B^3$.  For each $i>1$, the tangle $T_i\subset B_i = B_{i-1}\cup h_i$  is a marked tangle sum $T_{i-1}+ R_i $.  This sum is simple since the condition that attaching disks of $1$-handles in $H\times I$ meet two marking arcs shows that the gluing disk for the sum contains at least two tangle endpoints.  Thus $T_i\subset B_i$ is simple by \propref{gluing}a.  Since both summands are elementary, so is $T_i\subset B_i$.  It follows by induction that $T_n\subset B_n$, which is just $\R_K\subset B^3$, is simple and elementary.

That $\R_K$ has $4e$ strands follows from the observation that each $1$-handle in $H\times D^1$ intersects $4$ strands, and every strand intersects exactly one such $1$-handle.
\end{proof}


\section{Equivariant Heegaard splittings} \label{S:heegaard}
Given a closed $3$-manifold $\M$ with an action of a finite group $G$, we seek to replace $\M$ with a hyperbolic manifold with a $G$-action.  We may assume without loss of generality that $\M$ is connected.  The strategy is to find a $G$-equivariant Heegaard splitting $\H_0\cup\H_1$ of $\M$ (the goal of this section), and then to replace each handlebody $\H_i$ with a fake handlebody $\rH_i$ with a $G$-action, chosen so that the glued up manifold $\rH_0\cup\rH_1$ is hyperbolic (the goal of \S4).  This replacement process will require a further decomposition of the $\H_i$ into 0 and 1-handles that will be regarded as part of the structure of the Heegaard splitting.  Our exposition will be facilitated by passing back and forth between $M$ and its quotient $M/G$, and so for clarity and notational economy we henceforth denote the image of any subset $K$ of $M$ under the quotient map $M\to M/G$ by $\overline K$.  In particular $\MG = \M/G$.

If $G$ acts freely, then we could simply lift a Heegaard splitting of the quotient manifold $\MG$ with an arbitrary handle structure on the two sides.  When $G$ has fixed points, the quotient $\MG$ is an orbifold, albeit a {\sl good} one, and so its underlying space is still a 3-manifold.  In this case we will need the Heegaard splitting of $\MG$ {\sl and} the associated handle structures of the sides to be adapted to the orbifold structure, cf.\ \cite{mccullough-miller-zimmermann:handlebodies,zimmermann:genus} for a related discussion of {\it orbifold handlebodies}.  This splitting is constructed as follows. 

The $G$-action on $\M$ is locally linear since it is smooth, and orientation preserving by hypothesis.  Thus the stabilizer $G_x$ of any point $x\in\M$ is isomorphic to a finite subgroup of $\sothree$, so is either cyclic, dihedral, or one of the three symmetry groups of the Platonic solids, acting linearly on a 3-ball about $x$.  It follows that the {\it singular set} $\Si$ of all points in $\M$ with nontrivial stabilizers forms a graph in $\M$, which may include edges with endpoints identified and circle components with no vertices.  The vertices of $\Si$ are the points with noncyclic stabilizers, and each (open) edge is made up of points with the same nontrivial cyclic stabilizer.  To record this fact more precisely, we assign labels to these vertices and edges.  Since the noncyclic finite subgroups of $\sothree$ are all triangle groups \,(the dihedral group $D_{2n}$ is $\Delta(2,2,n)$, while the tetrahedral, octahedral and icosahedral groups are respectively $\Delta(2,3,3)\cong A_4$, $\Delta(2,3,4) \cong S_4$ and $\Delta(2,3,5) \cong A_5$), assign the integer triple $(p,q,r)$ to each vertex $x$ of $\Si$ with $G_x \cong \Delta(p,q,r)$, and assign the integer $n$ to any edge whose stabilizer is isomorphic to $C_n$.  

Now consider the image $\SG$ of $\Si$ in the quotient orbifold $\MG$.  This is also a graph, with labels inherited from $\Si$.  Since it is locally the singular set of a finite linear quotient of the 3-ball, $\SG$ is in fact a {\sl trivalent} graph, with each vertex labeled by the triple of labels on the edges incident to it.  We call $\SG$ the {\it branch locus} as $\M$ is the branched cover of $\MG$ along $\SG$, with branching indices given by the labels.  The quotient map $\Si\to\SG$ is illustrated in \figref{branch} near a tetrahedral vertex $v$ in $\Si$, shown on the left side of the figure.  The red, green and blue edges end at the vertices, edge midpoints, and face centers of a spherically inscribed regular tetrahedron, whose first barycentric subdivision is shown.  The $1$-skeleton of this subdivision restricts to a $1$-complex $K$ on the boundary of the ball, shown in black, whose projection $\K$ is an equatorial triangle on the right side of the picture.  The same phenomenon holds for vertices of any type $(p,q,r)$.

\fig{110}{FigBranch}{
\put(-245,-10){\small$\Si\subset\M$}
\put(-290,55){\small$K$}
\put(-66,-10){\small$\SG\subset\MG$}
\put(4,52){\small$\K$}
\put(-23,35){\small$3$}
\put(-110,42){\small$3$}
\put(-46,70){\small$2$}
\put(-85,57){\small$(2,3,3)$}
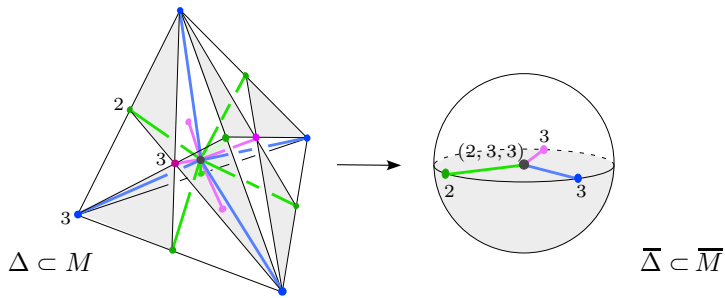
\caption{The picture of $\Si \to \SG$ near a point with tetrahedral stabilizer}
\label{F:branch}
\adjust[10]}

To build a Heegaard splitting of $\MG$ adapted to its orbifold structure, first extend $\SG$ to a larger trivalent graph $\DG\subset\MG$ whose complement is an open handlebody.  This is accomplished by adding new 1-labeled edges to $\SG$ corresponding to the 1-handles of a relative handlebody structure of the complement of a regular neighborhood of $\SG$, with endpoints chosen to lie at interior points of the edges in $\SG$.  Of course some edges of $\SG$ may be subdivided in this process.   If $e$ is such an edge with label $n$, then label each new edge of $\DG$ lying in $e$ with $n$, and each new vertex lying on $e$ with $(1,n,n)$ (specifying a cyclic stabilizer $C_n$).  Note that $\M$ is still a branched cover of $\MG$ along $\DG$, so we call $\DG$ the {\it extended branch locus}.   

Now let $\HG_0$ be a regular neighborhood of the extended branch locus $\DG$, built as a handlebody with 0 and 1-handles corresponding in the usual way to the vertices and edges of $\DG$.  The closure  $\HG_1$ of the complement of $\HG_0$ in $\MG$ is another handlebody of the same genus, which we can decompose into 0 and 1-handles using an arbitrarily chosen trivalent trivially-labeled spine (i.e.\ label all the edges with 1 and all the vertices with $(1,1,1)$).  This gives an  {\it orbifold Heegaard splitting} $\MG = \HG_0\cup\HG_1$ with $\DG\subset\HG_0$, in which each of the handlebodies is equipped with a specific handle structure reflecting the orbifold structure on $\MG$; we will refer to these as {\it orbifold handles}.   The lifts of the $\HG_i$ will then be equivariant handlebodies $\H_i$, equipped with their lifted handle structures, giving the desired equivariant Heegaard splitting $\M = \H_0\cup\H_1.$  In the next section, the orbifold Heegaard splitting will be used as a template to build a stabilized, equivariant fake Heegaard splitting $\rH_0 \cup \rH_1$ of the desired hyperbolic 3-manifold.


\section{Replacement handlebodies}\label{S:insert}
In this section, we describe how to insert equivariant simple doubly-slice tangles $\T_i$ into the handlebodies $\H_i$ in the equivariant Heegaard splitting $\M = \H_0\cup\H_1$ constructed in \S3; these will be called the {\it handlebody tangles} or {\it molecules}.  With the groundwork laid in \S2, the overall strategy is straightforward:  Place suitable equivariant $K$-molecules (as defined in \S2) in the $0$-handles of the decompositions of the $\H_i$ described above, and atomic tangles in the $1$-handles, in such a way that the markings match up to produce a simple marked tangle sum; the details are explained below.  The exteriors $\rH_i$ of the handlebody tangles $T_i$ in $\H_i$ are viewed as a replacements for the handlebodies $\H_i$, as in~\cite{ruberman:seifert}. Each $\rH_i$ is a simple and therefore (assuming positive Gromov norm) hyperbolic homology handlebody that comes equipped with an equivariant invertible cobordism from a genuine handlebody.  In \S5 we will show how to glue the $\rH_i$ together, and complete the proof of \thmref{hypthm}.

To achieve equivariance, it is convenient to work downstairs in the orbifold $\MG$ and then lift all of the constructions back up to $M$.  During this process, it should be noted that the tangles placed in the handles in $\MG$ are disjoint from the branch locus, so the lifted tangles in $M$ will be disjoint from the singular set.
\adjust[5]

\head{Tangles in the orbifold handles}

Recall from \S3 that the handlebodies $\HG_i$ in the orbifold Heegaard splitting $\MG = \HG_0\cup\HG_1$ are further decomposed into orbifold handles.  We give a procedure for inserting tangles into each of these handles.

The procedure is the same for each orbifold $1$-handle $h^1$, independent of its label $n$ which specifies the branching index:  Viewing $h^1$ as $D^2 \times D^1$ with orbifold singular set the $n$-labeled arc $\{0\} \times D^1$, insert a canonically marked $\R_8$ of index $(4,4)$, as explained in \defref{index} and shown in \figref{handletangles}b.  This tangle is simple and doubly slice by \propref{rub}, and by construction, disjoint from the branch locus.  
Its preimage in $\M$ is an $\R_{8n}$ of index $(4n,4n)$ in the $1$-handle lying above $h^1$, by \remref{lifts}.  

The procedure is also the same for each orbifold $0$-handle $h^0$, independent of the label $(p,q,r)$ which specifies the stabilizer of the associated vertex $v$ of the extended branch locus $\DG$ (a triangle or cyclic group): The handle $h^0$ is a $3$-ball centered at $v$ whose boundary $2$-sphere intersects $\DG$ in three points.  These points can be taken to lie on an equator, which is thereby subdivided into a spherical triangle $\K$ as shown on the right side of  \figref{branch}.  Now insert a $\K$-molecule $\R_{\K}$ in $h^0$, as explained in \defref{kmolecule}.  Recall that the construction of $\R_{\K}$ is guided by a $2$-dimensional handlebody $H$ in $\del h^0$, whose $0$-handles (one about each vertex in $\K$) are exactly the disks where the orbifold $1$-handles are attached.  By \lemref{kmolecule}, the tangle $\R_{\K}$ is simple and doubly-slice with twelve strands, four straddling the boundary of each of these attaching disks.  We can therefore arrange for the markings on the orbifold $0$ and $1$-handles to match up where they intersect. Once again, the tangles are disjoint from the branch locus.

\head{Assembling the orbifold tangles \ $\TG_i\subset\HG_i$}

The tangles $\TG_i\subset\HG_i$ are formed by gluing together the tangles in the orbifold $0$ and $1$-handles in $\HG_i$.  Thus when we attach the orbifold $1$-handles to the orbifold $0$-handles, we are performing simple tangle sums.  The final result is thus a pair of doubly-slice simple tangles $\TG_i\subset \HG_i$ for $i=0,1$, by Propositions \ref{P:ds}a and \ref{P:gluing}a.  The exterior of $\TG_i$ in $\HG_i$, denoted $\rHG_i$, is the homology handlebody that replaces $\HG_i$. 

\begin{remark}\label{R:genus}
Each $0$-handle in $\HG_i$ contributes $12$ components to $\TG_i$.  If the $\HG_i$ have genus $g$, then there are $2(g-1)$ such $0$-handles, corresponding to the vertices of the trivalent graph whose thickening is $\HG_i$.  Thus $\TG_i$ has $24(g-1)$ components.
\end{remark} 

\head{Lifting the orbifold tangles to \ $\T_i\subset\H_i$}

When the orbifold tangles $\TG_i\subset\HG_i$ are lifted to equivariant tangles $\T_i\subset\H_i$, the picture is exquisitely embellished, as in the creation of folded paper sculptures.  Fortunately, the proof that these lifted tangles are simple and doubly-slice is essentially the same as in the orbifold case.  As noted above, the model for the $1$-handles above a degree $n$ orbifold $1$-handle is an $\R_{8n}\subset D^2\times D^1$ of index $(4n,4n)$.  For the $0$-handles, the triangle $\K$ lifts to a $1$-complex $K$ (namely the $1$-skeleton of the first barycentric subdivision of the corresponding dihedron, tetrahedron, octahedron or icosahedron) and the tangle in the ball upstairs is a $K$-molecule by \remref{kmolecule}c.  The tangles now assemble into a pair of simple doubly-slice tangles $\T_i\subset\H_i$ for $i=0,1$, whose exteriors $\rH_i$ are the homology handlebodies that replace $\H_i$. 

\begin{remark}  
Many other tangles in the $3$-ball could be used in place of the atomic tangle $\R_4$ (and its branched covers $\R_{4n}$) to construct equivariant simple doubly-slice tangles in the handlebodies of $M$.  All that is required is that this tangle should be simple and doubly-slice, and that all its cyclic branched covers (along a suitable diameter of the $3$-ball) should also be simple and doubly-slice.
\end{remark}

\setcounter{thm}{0}


\section{Gluing replacement handlebodies and the proof of \thmref{hypthm}}\label{S:replacement}

Recall the statement of the main theorem:

\begin{thm}\label{T:hypthm}
Any closed $3$-manifold $\M$ equipped with 
an action of a finite group $G$ is equivariantly invertibly $\bz[\pi_1(\M)]$-homology cobordant to a hyperbolic $3\text{-manifold}$ $N$ with 
a $G$-action by isometries.   This cobordism may be chosen to be a product along a neighborhood of the singular set of the action. 
 \end{thm}

\begin{proof} 
Since there is an equivariant invertible $\bz[\pi]$-homology cobordism to a $3$-manifold with non-zero Gromov-norm by \lemref{gn}, we may assume that the manifold $M$ has positive Gromov norm.

In the last section, we constructed an equivariant handlebody decomposition $ \H_0 \cup \H_1$ of $M$ by lifting an orbifold handle decomposition $\HG_0\cup\HG_1$ of the quotient $\MG$. Then we removed neighborhoods of doubly-slice simple tangles $\TG_i\subset\HG_i$ and their lifts $\T_i\subset\H_i$ to obtain the replacement homology handlebodies $\rHG_i\subset \HG_i$ covered by $\rH_i\subset\H_i$.  The remaining step is to describe how to glue these replacement handlebodies together to build $N$, and how to create the equivariant homology cobordism that proves the theorem.

We begin by working in the quotient orbifold.  Recalling that $\TG_i$ is doubly-slice, choose a trivial tangle $\UG_i$ in $\HG_i$ that is invertibly concordant to $\TG_i$.  Removing a neighborhood of $\UG_i$ from $\HG_i$ has the effect of stabilizing $\HG_i$, i.e.\ adding $1$-handles to $\HG_i$.  For our purposes, and in particular to properly specify how to glue $\rHG_0$ and $\rHG_1$ together to form $N$, we need to make this more precise.

From this data and an ordering of the $n = 24(g-1)$ (by \remref{genus}) components of $\UG_i$ and $\TG_i$ consistent with their identification by the concordance, \lemref{longitude} gives rise to preferred decompositions
$
\del\rHG_i = \del\HG_i \cs n (S^1 \times S^1)
$
where the $k$th  torus summand is chosen so that the first $S^1$ factor is identified with the preferred longitude of the $k$th component of $\TG_i$, while the second $S^1$ factor is the meridian of that component.  From this, 
the boundary of the vertical part of the exterior of the tangle concordance in $\HG_i \times I$ acquires a preferred diffeomorphism with 
$
(\del\HG_i\times I) \cs_I (S^1 \times S^1 \times I)
$
where $\cs_I $ denotes the connected sum along a vertical arc. 

Now we glue $\rHG_0$ to $\rHG_1$ via a diffeomorphism of their boundaries that identifies corresponding tori in such a way that their meridians and longitudes are interchanged, but that is otherwise the identity.  This yields an orbifold $\MGh$.  Because of our choice of the meridian/longitude pair, a similar construction with $\TG_i$ replaced by $\UG_i$ simply stabilizes the orbifold Heegaard splitting of $\MG$, and does not change the resulting orbifold.  Gluing the exterior of the concordance $\calc_0$ in $\HG_0 \times I$ to the exterior of $\calc_1$ in $\HG_1 \times I$ then gives an orbifold homology cobordism $\PG$ from $\MG$ to $\MGh$.  By repeating the construction with the inverses of the concordances $\calc_i$ to obtain the inverse orbifold homology cobordism $\QG$ and applying \lemref{H-cob} we see that $\PG$ is in fact an invertible cobordism. 

Finally we pass to the orbifold covers.  The cobordism that has been constructed is automatically invertible and equivariant, and so it remains to show that the orbifold cover $\Mh$ is hyperbolic, with $G$ acting by isometries, and to check the homological properties of the invertible cobordisms $P$ and $Q$.  That $\Mh$ is hyperbolic follows from the fact that the tangles $\T_i\subset\H_i$ are simple, as noted at the end of \S4, together with \propref{gluing}b.   We know $N$ is either Seifert fibered or hyperbolic. However collapsing the fibers shows that a Seifert fiber space has zero Gromov norm~\cite{gromov:norm,yano:gromov}. The map $N\to P\cup_NQ \to I\times M \to M$ would then be a degree one map contrary to the fact that $M$ has positive Gromov norm. By Theorem 2.16, we may assume that $G$ acts by isometries so $\overline{N}$ is a hyperbolic orbifold and the entire construction is equivariant.

The cobordism $P$ is obtained by gluing relative $\bz[\pi_1(M)]$ homology cobordisms between the replacement handlebodies $\rH_i$ and $H_i - U_i$. Because of the interchange of meridian and longitude, the gluing maps are compatible and so there is a map between the Mayer-Vietoris sequence (with coefficients in  $\bz[\pi_1(M)]$) for this decomposition of $M$ with the Mayer-Vietoris sequence for  $P$.  By the $5$-lemma, we see that the inclusion of $M$ into $P$ is a homology equivalence, so that $P$ is a $\bz[\pi_1(M)]$-homology cobordism. The same argument applies to the inverse cobordism $Q$. 
Since the cobordisms $P$ and $Q$ were computed by modifications to $I\times M$ away from the singular set we see that the cobordism is a product in this neighborhood.
\end{proof}


\section{Applications of hyperbolization}\label{S:corks}

This section supplies proofs for the corollaries of \thmref{hypthm} listed in the Introduction. 

\head{Hyperbolic $G$-corks}

As mentioned in the Introduction, our original motivation was to show the existence of (effective) hyperbolic $G$-corks, and we start there. 
\setcounter{thm}{1}
\begin{cor}\label{C:hypcork}
There exist hyperbolic $G$-corks for any finite group $G$ that acts 
 on the boundary of some compact contractible $4$-dimensional submanifold of $\br^4$.
\end{cor}
\begin{proof}
The main result of our earlier paper~\cite{auckly-kim-melvin-ruberman:equivariant-cork} asserts that if $G$ is a finite group that acts smoothly on the boundary of some compact contractible 4-dimensional submanifold of $\br^4$, then there exists a 4-manifold $X$ and a compact contractible submanifold $C\subset X$, with a $G$-action on its boundary, such that the 4-manifolds
$
X_{C,g} = (X - \inte C) \cup_g C
$
for $g\in G$ are all smoothly distinct.  We say that $C$ (with its boundary $G$-action) is an effective $G$-cork in $X$.   Now by \thmref{hypthm}, there is a $G$-equivariant invertible homology cobordism $P$ from $\del C$ to a hyperbolic homology sphere $N$, with inverse the equivariant homology cobordism $Q$.  \\[1ex]
{\bf Claim.} $C' = C \cup_{\del C} P$ is an effective $G$-cork in $X$, with boundary $\del C' = N$.  
\\[1ex]
To see this, note first that 
$$
C' = \left(C \cup_{\del C} P\right) \subset \left(C \cup_{\del C} P \cup_N Q\right) \cong C,
$$
and this induces an embedding of $C'$ in $X$.   Now the $G$-equivariance of $P$ and $Q$ implies $X_{C',g} \cong X_{C,g}$ for every $g \in G$.  Since the manifolds $X_{C,g}$ are smoothly distinct as $g$ runs over $G$, the same is true for $X_{C'\!\!,\,g}$.
\end{proof}

\head{Non-extendible group actions}

The hyperbolization results in~\cite{Myers2,ruberman:seifert} have been used to show that results proved about homology cobordism and knot concordance can apply to hyperbolic examples; the next application is an equivariant version of this principle.  Recall that for pairwise relatively prime $(a,b,c)$, the Brieskorn homology sphere $\Sigma(a,b,c)$  is the link of the complex singularity 
$
x^a + y^b + z^c = 0.
$
It admits a fixed point free circle action $t (x,y,z) = (t^{bc}x, t^{ac} y, t^{ab}z)$, so is Seifert fibered. If $p$ is a prime that does not divide $a$, $b$ or $c$, then the $\bz_p$ contained in $S^1$ acts freely on $\Sigma(a,b,c)$.  Work of Kwasik-Lawson~\cite{kwasik-lawson:contractible} and Anvari-Hambleton~\cite{anvari-hambleton:contractible} shows that this $\bz_p$ action does not extend over any contractible manifold with boundary $\Sigma(a,b,c)$.  (Note that while not all Brieskorn spheres bound contractible manifolds--see for instance \cite{fs:pseudofreeorbifolds,furuta:cobordism}, there are infinite families \cite{casson-harer,stern:brieskorn,fickle:contractible} that do.) We now show that \thmref{hypthm} gives many hyperbolic homology spheres with free $\bz_p$ actions satisfying this property.
\setcounter{thm}{2}
\begin{cor}
For any Brieskorn homology sphere $\Sigma(a,b,c)$ and prime $p \nmid abc$, there is a hyperbolic homology sphere $N(a,b,c)$ with a free $\bz_p$-\,action that is $\bz_p$-\,equivariantly homology cobordant to the standard action on $\Sigma(a,b,c)$, and that does not extend smoothly  over any contractible $4$-manifold that $N(a,b,c)$ might bound.   
\end{cor}
We remark that for many choices of $(a,b,c)$, the conclusion can be strengthened to say that the action of $\bz_p$ does not extend smoothly over any {\em acyclic} $4$-manifold that $N(a,b,c)$ might bound. This is shown via the method of Kwasik and Lawson taking into account that Donaldson's definite manifolds theorem applies to non-simply connected manifolds; see~\cite{donaldson:orientation}. Kwasik and Lawson~\cite[Proposition 12]{kwasik-lawson:contractible} give a list of examples to which this method applies.  Anvari and Hambleton~\cite[Theorem 4.4]{anvari-hambleton:contractible} also give a non-extension result over acyclic manifolds under the additional assumption that the fundamental group of the boundary normally generates that of the acyclic manifold.
\begin{proof}[Proof of \coref{hyperextend}]
The condition that $p$ does not divide $abc$ implies that the action of $\bz_p$ on $\Sigma(a,b,c)$ is free. By \thmref{hypthm} there is a $\bz_p$-equivariant invertible homology cobordism $P$ from $\Sigma(a,b,c)$ with its standard $\bz_p$-\,action to a hyperbolic manifold $N =N(a,b,c)$.  By construction, the action of $\bz_p$ on $P$ is free. If the action on $N$ extends over a contractible manifold $W$, then the manifold
$
W \cup_N P
$
is a homology ball over which the $\bz_p$-\,action on $\Sigma(a,b,c)$ extends. By \propref{normal}, this homology ball is simply connected, and hence contractible, contradicting~\cite{anvari-hambleton:contractible}.
\end{proof}

\head{Acyclic versus contractible}

An important consequence of Taubes' periodic ends theorem~\cite{taubes:periodic}, observed by Akbulut, is that there are reducible homology spheres that bound homology balls, but do not bound contractible manifolds; the original example was $\Sigma(2,3,5) \cs - \Sigma(2,3,5)$. We show that one can in fact choose the homology sphere to be hyperbolic.

\begin{cor}\label{C:acyclic}
There are infinitely many hyperbolic homology spheres that bound homology balls but do not bound contractible manifolds.
\end{cor}
\begin{proof} 
We first describe how to generate one example, and then describe the modifications required to construct infinitely many. 
Let $\Sigma$ be any homology sphere that bounds a simply connected smooth $4$-manifold $X$ with non-standard  definite intersection form, for example the Poincar\'e homology sphere. Note that  $\Sigma\cs-\Sigma$ is the boundary of a homology ball, namely $I\times  (\Sigma - \inte B^3)$.  By our main theorem, there is an invertible homology cobordism $P$ 
from $\Sigma\cs-\Sigma$ to a hyperbolic $3$-manifold $N$.  Furthermore, by \propref{normal} the fundamental group of $P$ is normally generated by the fundamental group of $N$.  One sees that 
$
W = I\times  (\Sigma - \inte B^3) \cup_{\Sigma\cs-\Sigma} P
$
is a homology ball with boundary $N$.

Now assume that $N$ bounds a contractible manifold $Z$. Adding a $3$-handle to $Z\cup_N P$ along the sphere separating $\Sigma$ and $-\Sigma$ results in a simply connected (since $\pi_1(P)$ is normally generated by $\pi_1(N)$) homology cobordism $V$ with boundary $\Sigma \cup -\Sigma$.  This contradicts~\cite[Proposition 1.7]{taubes:periodic}, a consequence of Taubes' periodic ends theorem.  

To show that there are infinitely many distinct examples, we iterate this process.  Let $X_k$ denote the boundary connected sum of $k$ copies of $X$; its intersection form is also non-standard, as is readily verified using Elkies' criterion for diagonalizablity of a unimodular form~\cite{elkies:lattice}. Now construct an invertible homology cobordism $P_k$ between $\cs^k (\Sigma \cs -\Sigma)$ and  $\cs^k N$ by taking the connected sum of $k$ copies of $P$ along arcs running from $\Sigma \cs -\Sigma$ to $N$.  

Again applying our main theorem, there is an invertible homology cobordism $P_k'$ from $\cs^k N$  to a hyperbolic manifold $N_k$.  
Note that there is a degree one map from $N_k$ to $\cs^k N$, so that 
$$
\vol(N_k) \geq k\cdot \vol(N).
$$
It follows that the Gromov norms of the $N_k$ are unbounded, so an infinite sequence of them are distinct. Glue $P_k$ to $P_k'$ along $\cs^k N$  to obtain an invertible homology cobordism from $\cs^k (\Sigma \cs -\Sigma)$ to $N_k$. As above,  \propref{normal} says that the fundamental group of this cobordism is normally generated by $\pi_1(N_k)$.  

By construction, $N_k$ is the boundary of a homology ball. If it bounds a contractible manifold $Z$, we proceed as before. Add a single $3$-handle along a 2-sphere in $\cs^k (\Sigma \cs -\Sigma)$ that separates the $k$ copies of $\Sigma$ from the copies of $-\Sigma$ to obtain a cobordism from $\cs^k \Sigma$ to $\cs^k \Sigma \coprod N_k$. Gluing $Z$ to $N_k$ results in a simply connected homology cobordism from $\cs^k \Sigma$ to $\cs^k \Sigma$, which again contradicts Taubes' theorem.
%
\end{proof}

\head{Finite groups acting on homology spheres}\label{S:spaceforms}

Finally, we apply \thmref{hypthm} to questions related to the classical spherical space form problem; see the Davis-Milgram survey \cite{davis-milgram:survey} and the discussion of problem 3.37 in Kirby's problem list \cite{kirby:problems96}.  
\begin{cor}\label{C:hyphomsphere}
Any finite group that acts on a homology $3$-sphere also acts on a hyperbolic homology $3$-sphere, with equivalent fixed-point behavior $($i.e.\ the two actions are equivariantly diffeomorphic near their fixed point sets$)$.  In particular, there exist finite groups that are not subgroups of $SO(4)$ $($so by geometrization do not act freely on $S^3)$ that act freely on some hyperbolic homology $3$-sphere. \end{cor}

\noindent{\bf Remark.}
Presumably there are {\it infinitely} many finite groups that satisfy the final conclusion of the corollary.  This would follow either from Theorem 6.10 in~\cite{davis-milgram:survey} (stated there without proof) or the Generalized Riemann Hypothesis (as noted in the proof below).

\begin{proof}[Proof of \coref{hyphomsphere}]
The first part is a direct corollary of \thmref{hypthm}, replacing an action of a group $G$ on a homology sphere by an action on a hyperbolic homology sphere. The second part, constructing free actions on homology spheres by groups that cannot act freely on the $3$-sphere, requires results on the topological spherical space form problem dating to the 1970s and 1980s.  The underlying principle is that there are homotopy-theoretic (finiteness) and surgery-theoretic obstructions, depending only on $n$ modulo $8$, for a finite group $G$ to act freely on a sphere of dimension $n$. If $n$ is greater than $4$, the vanishing of these obstructions is sufficient for the existence of such an actions.  In dimension $3$, the vanishing of the obstructions implies only that $G$ acts freely on a homology $3$ sphere; see for example~\cite[Remark 8.2]{hambleton:space-forms}.

Madsen \cite{madsen:spaceforms}, Milgram~\cite{milgram:swan}, and Bentzen~\cite{bentzen:space-forms} evaluated the finiteness and surgery obstructions in number theoretic terms.  Their results show that many generalized quaternionic groups $Q(8p,q)$ (in fact {\it infinitely} many if one assumes the Generalized Riemann Hypothesis, the smallest being $Q(24,313)$~\cite{bentzen:space-forms}) act freely on spheres in dimensions $8k+3$ ($k> 0$) and hence on homology $3$-spheres.   
We deduce directly from~\thmref{hypthm} that such homology spheres can be taken to be hyperbolic.
\end{proof}

\noindent{\bf Remark.} The constructions that produced nonlinear finite group actions on homology $3$-spheres gave in principle a description of these homology spheres as surgeries on links. The complexity of the necessary calculations in surgery theory, however, did not yield any insight as to the geometric nature of these homology spheres.  The geometrization theorem~\cite{perelman:ricci,perelman:extinction,perelman:surgery,morgan-tian:geometrization} shows that they cannot be the $3$-sphere.  It would be of interest to know if these groups can act freely on a Seifert fibered homology sphere.


\vskip 15pt 
\noindent {\large \bf Appendix \ \   Antisymmetry of invertible cobordism of $3$-manifolds}\label{A:antisymmetric}

\vskip 10pt

We show that for closed oriented $3$-manifolds, {\it invertible cobordism} is an antisymmetric relation, and thus a partial order.\foot{The first version of this article established antisymmetry of invertible {\it homology} cobordism but the additional hypothesis is not necessary.} 
This is false in higher odd dimensions, as seen from the existence of h-cobordisms $X$ with non-trivial Whitehead torsion, for which $-X$ is the inverse cobordism; compare~\cite[Lemma 7.8]{rourke-sanderson:book}. 

First a trivial observation\,:  Any invertible cobordism between manifolds decomposes as a disjoint union of invertible cobordisms between their connected components.  Thus for our present purposes, we may implicitly assume  that all manifolds are connected.  We will also make use of the following presumably well known fact about $3$-manifolds, which we could not find explicitly in the literature.
 
\begin{lemma}\label{L:homotopy}
 Let $f\colon M \to N$ be a degree-one map of oriented $3$-manifolds that induces an isomorphism on fundamental groups. Then $f$ is a homotopy equivalence.
\end{lemma}

\begin{proof}
Let $\pi$ denote the fundamental group of the two manifolds. By Whitehead's theorem, it suffices to show that the induced map on  homology with coefficients in $\bz[\pi]$ is an isomorphism. This is clear for $H_0$ and $H_1$. If $\pi$ is infinite, then $H_3 =0$ and there is nothing to prove, and if $\pi$ is finite then this follows since $f$ has degree $1$.  This leaves dimension $2$.  Poincar\'{e} duality with $\bz[\pi]$ coefficients says that there is a commutative diagram with both horizontal arrows isomorphisms.  The diagram commutes because the degree of $f$ is $1$. 
\[
\xymatrix{
H^1(M;\bz[\pi]) \ar[r]^{\cap [M]}& H_2(M; \bz[\pi])\ar[d]^{f_*}\\
H^1(N;\bz[\pi]) \ar[r]^{\cap [N]}\ar[u]^{f^*} & H_2(N; \bz[\pi])
}
\]

For any CW complex $X$ with fundamental group $\pi$, we have 
$$
H^1(X;\bz[\pi)]) \cong H^1(\pi;\bz[\pi]) = H^1(K(\pi,1);\bz[\pi]).
$$
The standard argument for this is to create a $K(\pi,1)$ by adding cells of dimension $3$ or greater to $X$; the resulting inclusion map $X \hookrightarrow K(\pi,1)$ induces an isomorphism on $H^1$. If $f: X \to Y$ is a map inducing an isomorphism on $\pi_1$ then there is an automorphism $g$ of $\pi$  making the following diagram commute.
\[
\xymatrix{
 H^1(\pi;\bz[\pi]) \ar[r]  & H^1(Y;\bz[\pi)]) \\
 H^1(\pi;\bz[\pi]) \ar[r]\ar[u]^{g^*}  & H^1(X;\bz[\pi)]) \ar[u]_{f^*} 
}
\]
It follows that $f^*$ is an isomorphism.
\end{proof}

\begin{theorem}\label{T:antisymm}
Let $M$ and $N$ be closed $3$-manifolds. If there is an invertible cobordism from $M$ to $N$, and one from $N$ to $M$, then $M$ and $N$ are homeomorphic.
\end{theorem}

\begin{proof}
Let $P$ be the cobordism from $M$ to $N$, and $Q$ be the inverse cobordism from $N$ to $M$, so that $P \cup_N Q = M \times I$.  This gives a map $f\colon N\to M$, the composition of the inclusion $N\hookrightarrow M\times I$ followed by the projection $M\times I \to M$.  There is also another pair of cobordisms $Q'$ from $N$ to $M$ and $P'$ from $M$ to $N$, so that  $Q'  \cup_M P' = N \times I$, and this gives a map $g\colon M\to N$.
Both of these maps have degree one, so their induced maps on $\pi_1$ are surjective. Thus the composition 
$
g_* \circ f_* 
\colon  \pi_1(N) \to \pi_1(N)
$
is surjective.  But 3-manifold groups are Hopfian~\cite{aschenbrenner-friedl-wilton:book}, which means that in fact this composition is an isomorphism. It follows that $f_*$ is injective, so it is an isomorphism. We write $\pi$ for $\pi_1(M) \cong \pi_1(N)$.

Computing the fundamental group of $M \times I = P \cup_N Q$ by van Kampen's theorem yields a pushout diagram:
\[
\label{pushout}
\begin{gathered}
\xymatrix{
& \pi_1(Q) \ar[dr]^{j_Q}&&\\
\pi_1(N) 
\ar[rr]^{f_*} 
\ar[ur]^{i_Q}
\ar[dr]_{i_P} && 
\pi_1(M) &\hskip-25pt\cong \pi_1(M\times I) 
\\
& \pi_1(P) \ar[ur]_{j_P}&&
}
\end{gathered}
\tag{$*$}
\]
Since $f_*$ is an isomorphism, $i_P$ and $i_Q$ are injective, and $j_P$ and $j_Q$ are surjective. A standard result about pushouts~\cite[Theorem IV.2.6]{lyndon-schupp} says that in fact $j_P$ and $j_Q$ are injective, so all of these maps are isomorphisms. In particular, \lemref{homotopy} implies that $f$ is a homotopy equivalence. Denote the fundamental group of all of these manifolds by $\pi$.

Now we claim that $i_P$ and $i_Q$ are homotopy equivalences, or equivalently, induce isomorphisms on homology with coefficients in $\bz[\pi]$. Poincar\'e duality, as in \lemref{homotopy}, implies that the inclusions of $M$ into $P$ and $Q$ are homotopy equivalences, and so both manifolds are h-cobordisms. To see this, we use the Mayer-Vietoris sequence, a portion of which is drawn in the following diagram; the twisted coefficients in $\bz[\pi]$ are understood.
\begin{equation*}\label{mv}
\xymatrix@R-18pt{
&& H_k(Q) \ar[dr]^{j_Q}&&\\
\cdots\ \ar[r] & H_k(N) 
\ar[ur]^{i_Q}\ar[dr]_{i_P}&  \bigoplus & 
H_k(M) &\hskip-25pt\cong H_k(M\times I) \ar[r] & \ \cdots 
\\
&& H_k(P) \ar[ur]_{j_P}&&
}
\end{equation*}
Since $j_Q \circ i_Q \simeq f$, we see that $i_Q$ is injective, and $j_Q$ is surjective, so the Mayer-Vietoris sequence splits into a sum of short exact sequences.  An easy diagram chase shows that $i_Q$, $i_P$, $j_Q$, and $j_P$ are all isomorphisms. In particular, $P$ (and all the other cobordisms) is an h-cobordism.  A theorem of Kwasik-Schultz~\cite{kwasik-schultz:torsion} says that the Whitehead torsion of $P$ must vanish. In particular, all of the inclusion maps are simple homotopy equivalences, and so $f$ is a simple homotopy equivalence. Finally, Turaev proved~\cite{turaev:geometric} that this implies that $f$ is homotopic to a homeomorphism.\end{proof}

We remark that if $M$ and $N$ are hyperbolic manifolds, then there is an alternate (and perhaps simpler) route to this conclusion, based on the Gromov-Thurston proof of Mostow's rigidity theorem~\cite[Theorem 6.4]{thurston:notes2002} as well as~\cite{gromov:norm,haagerup-munkholm:simplices}. This proof implies directly that if there are degree one maps from $M$ to $N$ and from $N$ to $M$, then $M$ and $N$ are homeomorphic.  Such maps are constructed in the first paragraph of the proof above; in this setting one doesn't have to establish that the maps are homotopy equivalences. 
\begin{remark*}
As the proof of \thmref{antisymm} indicates, the study of invertible cobordisms has a large overlap with the study of degree one maps. For instance, one says that $M$ dominates $N$ if there is a map of non-zero degree from $M$ to $N$. A recent result of Liu~\cite{liu:domination} says that a given closed $3$-manifold dominates only finitely many other $3$-manifolds. It follows {\em a fortiori} that the same result holds for the ordering coming from invertible cobordism. 
\end{remark*}

Finally, we establish the following general property of maps induced on the fundamental group of invertible cobordisms that was used in Corollaries \ref{C:hyperextend} and \ref{C:acyclic}. 

\begin{proposition}\label{P:normal}
Let $P$ be an invertible cobordism from $M$ to $N$, with inverse cobordism $Q$. Then the image of the map $i_P$ induced by the inclusion of $N$ into $P$ normally generates $\pi_1(P)$, and likewise the image of $i_Q$ normally generates $\pi_1(Q)$.
\end{proposition}

\begin{proof}
We adopt the notation of the previous proof.  It suffices to show that the quotient groups
$$
G_P \ = \ \pi_1(P)/\langle \im(i_P) \rangle  
\qquad\text{and}\qquad     
G_Q \ =  \ \pi_1(Q)/\langle \im(i_Q) \rangle 
$$
\vskip-10pt 
\noindent are trivial, where $\langle \  \rangle$ denotes normal closure. The natural maps $\pi_1(P) \overset{k_P}{\to} G_P * G_Q \overset{k_Q}{\from}\pi_1(Q)$ induce a unique map $h:\pi_1(M) \to G_P * G_Q$ with $h\circ j_P = k_P$ and $h\circ j_Q = k_Q$, extending the diagram  $(*)$ above, which must be trivial since $f_*$ is onto.  This forces $k_P$ and $k_Q$ to be trivial, which implies that $G_P$ and $G_Q$ are trivial.
\end{proof}

\bibliography{hyperbolic}

\end{document}